\documentclass[12pt]{amsart}

\usepackage{graphicx}
\usepackage{amssymb}
\usepackage{amsthm}
\usepackage{graphics}
\usepackage{amsmath}
\usepackage{epsfig}
\usepackage{amsfonts}
\usepackage{amssymb}
\usepackage{graphicx}
\usepackage{geometry}

\newtheorem{theorem}{Theorem}[section]
\newtheorem{lemma}[theorem]{Lemma}

\theoremstyle{definition}

\newtheorem{claim}[theorem]{Claim}
\newtheorem{conjecture}[theorem]{Conjecture}

\begin{document}

\title{On the transient number of a knot}

\author{Mario Eudave-Mu\~noz}
\address{ \hskip-\parindent
Mario  Eudave-Mu\~noz\\
Instituto de Matem\'aticas\\
 Universidad Nacional Aut\'onoma de M\'exico\\ 
 MEXICO}
 \email{mario@matem.unam.mx}
 
 \author{Joan Carlos Segura-Aguilar}
\address{ \hskip-\parindent
Joan Carlos Segura-Aguilar\\
Universidad Nacional Aut\'onoma de M\'exico\\ 
 MEXICO}
 \email{joancarlos28@gmail.com}



\maketitle

\begin{abstract}
The transient number of a knot $K$, denoted $tr(K)$, is the minimal number of simple arcs that have to be attached to
$K$, in order that $K$ can be homotoped to a trivial knot in a regular neighborhood of the union of $K$ and the arcs. 
We give a lower bound for $tr(K)$ in terms of the rank of the first homology group of the double branched cover of $K$.
In particular, if $t(K)=1$, then the first homology group of the double branched cover of $K$ is cyclic. Using this,
we can calculate the transient number of many knots in the tables and show that there are knots
with arbitrarily large transient number. 
\end{abstract}

\section{Introduction}.

Let $K$ be a knot in the 3-sphere and let $M$ be a submanifold of $S^3$ containing $K$. We say that $K$ is
transient in $M$ if $K$ can be homotoped within $M$ to the trivial knot in $S^3$; otherwise $K$ is called
persistent. For example, $K$ is persistent in a regular neighborhood $\mathcal{N}(K)$ of $K$, but it is transient in a 3-ball
$B$ containing $K$. Yuya Koda and Makoto Ozawa \cite{KO} proved that every knot is transient in a submanifold $M$ if 
and only if $M$ is unknotted, that is, its complement in $S^3$ is a union of handlebodies. Then Koda and Ozawa \cite{KO} 
introduced a new invariant of knots, called the transient number of $K$, which somehow measures, starting with 
$\mathcal{N}(K)$, how large must be a submanifold in which $K$ is transient. 

The transient number is defined as follows: given a knot $K$ in $S^3$, 
there is a collection of arcs $\{ \tau_1, \tau_2, \dots, \tau_n\}$, disjointly embedded in $S^3$, each $\tau_i$ 
intersecting $K$ exactly at its endpoints, such that $K$ can be homotoped in a regular neighborhood of $K$ union the arcs, 
$T=\mathcal{N}(K\cup \tau_1 \cup \dots \cup \tau_n)$, into the trivial knot.  That is, we perform crossing changes and
isotopies inside $T$, until we get the trivial knot $K'$. Note that any knot $K'$ obtained from $K$ in this way is not 
trivial in $T$, i.e. it cannot bound a disk contained in $T$, but it can be trivial in $S^3$. The transient number of $K$,
$tr(K)$, is then defined as the minimal number of arcs needed in such a system of arcs. The transient number is 
related to other knot invariants, namely $tr(K)\leq u(K)$, where $u(K)$ is the unknotting number and $tr(K)\leq t(K)$,
where $t(K)$ is the tunnel  number. It is easy to check these inequalities. For the unknotting number, given a sequence
of crossing changes that unknot $K$, consider for each crossing change an arc with endpoints in $K$ that guides
the crossing change, such that a regular neighborhood of the arc encapsulates the crossing change, then clearly 
$K$ can be made trivial in a neighborhood $T$ of $K$ union the arcs. 
For the case of tunnel number, consider a tunnel system and a
neighborhood $T$ of the union of $K$ and the arcs, such that the exterior is a handlebody. Isotope $T$ such that it looks 
like a standard handlebody in $S^3$. Then $K$ can be projected to the intersection of a plane with $T$, and
guided by this projection to a plane,
crossing changes can be performed to $K$ inside $T$ to get the trivial knot.

However a knot $K$ can have $tr(K)=1$, but $u(K)$ and $t(k)$ can be larger than one. Some examples with this
property are given in \cite{KO}. However, in that paper no example is given of a knot $K$ with $tr(K)>1$.
Homology groups of branched covers have been used to bound invariants like $u(K)$ and $t(K)$, which goes back
to the work of Wendt \cite{W}. In fact, it is well known that if $\Sigma[K]$ denotes the double branched cover of $K$, then
the rank of the group $H_1(\Sigma[K])$ gives a lower bound for $u(K)$, see \cite{W} or \cite{M}. It is also not
difficult to show that the rank of $H_1(\Sigma[K])$ is at most $2t(K)+1$; in particular it is known that if $t(K)=1$ then 
$H_1(\Sigma[K])$ is a cyclic group (though not explicitly stated, this follows from the computations
of homology of cyclic covers done in \cite{G}, or from \cite{Paz}). 

In this paper we prove that the rank of the first homology group of a branched cover of a knot give lower bounds for the
transient number. By using the Montesinos trick, it can be shown that if $K$ is a knot with $u(K)=n$, then $\Sigma[K]$
can be obtained by Dehn surgery on an $n$-component link in $S^3$, which implies then the bound for $u(K)$. 
In this paper we do a kind of generalized Montesinos trick. Our main results are the following.

\begin{theorem}\label{icubiertageneral} If $K$ is a knot in $S^3$ such that $tr(K) = n$, then the first homology group of the double branched cover of $K$ has a presentation with at most $2n + 1$ generators.
\end{theorem}


\begin{theorem}\label{icubiertageneral-p} If $K$ is a knot in $S^3$ such that $tr(K) = n$, then the first homology group of the $p$-fold branched cover of $K$ has a presentation with at most $pn + 1$ generators.
\end{theorem}

These results imply that $rank(H_1(\Sigma[K])\leq 2tr(K)+1$. Also, if 
$\Sigma_p[K]$ denotes the $p$-fold branched cover of $K$, it follows that $rank(H_1(\Sigma_p[K])\leq ptr(K)+1$. 

For the case that $tr(K)=1$, we can get a better bound. In fact,
by doing a careful calculation of the first homology group of $\Sigma[K]$, we get the following result.

\begin{theorem} \label{icubiertatransito} If $K$ is a knot in $S^3$ such that $tr(K) = 1$, then the first homology group of the double branched cover of $K$ is cyclic. \end{theorem}

Of course, these results may not be sharp. It would be interesting to find sharp bounds for these inequalities.
It would also be interesting to find bounds for the transient number depending on other classical invariants of knots.

Given any knot invariant, it is always interesting to study its behavior under connected sums of knots. We have the
following:

\begin{theorem} Let $K_1$, $K_2$ be knots in $S^3$. Then $tr (K_1\# K_2) \leq tr (K_1)+tr (K_2) + 1$. \end{theorem}

The paper is organized as follows. In Section \ref{branchedcovers} we sketch a proof that the unknotting number and 
tunnel number are bounded below by the rank of the first homology group of 
double branched covers. Then prove the main results. As part of the proofs, we show also that if $t(K)=1$,
then $H_1(\Sigma[K])$ is cyclic; this is used in the proof of Theorem \ref{icubiertatransito}. In Section \ref{examples}
we give examples of knots with large transient number and explore the transient number of knots in the tables ok KnotInfo
\cite{Knot}. In Section \ref{connectedsums} we consider the transient number of a 
connected sum of knots,  prove some facts and propose some problems. 

Through the paper we work in the piecewise linear category. To avoid cumbersome notation we 
use expressions like the double branched cover of a knot to mean the double cover of $S^3$ branched along the knot. If
$\Lambda$ is a simple closed curve in the boundary of a 3-manifold $M$, we say adding a 2-handle along $\Lambda$,
to mean that we attach a 2-handle $D^2\times I$ to $M$, such that $\partial D^2 \times I$ is identified with a regular
neighborhood of $\Lambda$ in $\partial M$, which is an annulus. Also, if $M$ and $T$ are compact 3-manifods,
with $T\subset M$, then by $M\backslash T$ we mean $M$ minus the interior of $T$, or well the closure in $M$ of $M-T$.
If $X$ is a topological space, $\vert X \vert$ denotes its number of components. 

\section{Transient number and double branched covers}\label{branchedcovers}

This section is inspired by an idea that is used to build the double branched cover of a knot with unknotting number 
equal to one. Consider a knot $K$ in $S^3$ with unknotting number equal to one. Let $\alpha$ be an arc embedded 
in $S^3$, with endpoints in $K$, such that a regular neighborhood of it encapsulates the crossing change. 
So there is an homotopy in $\mathcal{N}(K \cup \alpha)$ 
between the knot $K$ and the trivial knot, which is denoted by $K'$. Clearly this homotopy can be taken so that it is 
constant in $\mathcal{N}(K) \backslash\, \mathcal{N}(\alpha)$ and that the changes are occurring only in 
$\mathcal{N}(\alpha)$; so we assume that $K'$ is obtained from $K$ just by taking the
 two arcs $K \cap \mathcal{N}(\alpha)$ and passing one arc through the other, which would correspond to a crossing 
change in the corresponding knot diagram. Due to the above we have that 
$K\cap (S^3 \backslash \mathcal{N}(\alpha))=K' \cap (S^3 \backslash \mathcal{N}(\alpha))$.

Let $\Sigma (K')$ be the double branched cover of the knot $K'$, with covering
function given by $p : \Sigma(K') \rightarrow S^3$. Now, since $K'$ is the trivial knot, $\Sigma(K')$ is 
homeomorphic to $S^3$. We know that $\mathcal{N}(\alpha)$ is a 3-ball intersecting $K'$ in two arcs, therefore 
$p^{-1}(\mathcal{N}(\alpha))$ is a solid torus, and $p^{-1}(\partial \mathcal{N}(\alpha))$ is a surface of 
genus one. Therefore, $S^3 \backslash p^{-1}(\mathcal{N}(\alpha))$ is a double cover of 
$S^3 \backslash\, \mathcal{N}(\alpha)$ branched along $K \cap (S^3 \backslash \, \mathcal{N}(\alpha))$. So to finish building the double branched cover of the knot $K$, 
all we have to do is to refill $S^3 \backslash \, p^{-1}(\mathcal{N}(\alpha))$ appropriately.

Note that there exists a compressing disk for $\partial (\mathcal{N}(\alpha)) \backslash K$ contained in 
$\mathcal{N}(\alpha) \backslash K$; we denote this disk by $D$ (see Figure 12). 
As $K \cap D = \emptyset$ then $\vert K' \cap D\vert$ is an even number, so the curve
$\partial D$ is lifted by $p$ into two curves in $p^{-1}(\partial \mathcal{N}(\alpha))$; we
denote these curves by $\Lambda_1$ and $\Lambda_2$. Let $\Sigma'$ be the 3-manifold obtained by
adding two 2-handles to the 3-manifold $S^3 \backslash \, p^{-1}(\mathcal{N}(\alpha))$, attached along 
the curves $\Lambda_1$ and $\Lambda_2$; we denote these 
2-handles by $\overline{\Lambda}_1$ and $\overline{\Lambda}_2$ respectively. So 
$\Sigma' =[S^3 \backslash p^{-1}(\mathcal{N}(\alpha))]\cup[\overline{\Lambda}_1 \cup \overline{\Lambda}_2]$.

We know that $\Lambda_1 \cup \Lambda_2$ is a double cover of $\partial D$ with covering function given
by $p \vert_{\Lambda_1 \cup \Lambda_2}$. So we can extend the function $p \vert_{\Lambda_1 \cup \Lambda_2}$
to $\overline{\Lambda}_ 1\cup \overline{\Lambda}_2$, to get that $\overline{\Lambda}_1 \cup \overline{\Lambda}_2$ is a double cover of $\mathcal{N}(D)$. From this follows that $\Sigma'$ is a double cover of 
$[S^3 \backslash p^{-1}(\mathcal{N}(\alpha))]\cup \mathcal{N}(D)$ branched along two arcs of $K$.

We have that $\partial ([S^3\backslash \mathcal{N}(\alpha)]\cup [\mathcal{N}(D)])$ consists of two 2-spheres and 
$\partial \Sigma'$ also consists of two 2-spheres. Also, the 2-spheres of $\partial \Sigma'$ are a double cover of the 
two spheres of $\partial ([S^3\backslash \mathcal{N}(\alpha)]\cup \mathcal{N}(D)]$ branched over 
the points $K \cap \partial([S^3\backslash \mathcal{N}(\alpha)]\cup \mathcal{N}(D))$. 

Now we can fill the sphere boundary components of $\Sigma'$ with 3-balls, and extend the function $p$ to these 3-balls in order to get the double covering of $S^3$ branched along the knot $K$.

The idea described above is known as the Montesinos trick. Similar to the previous construction, 
we will build the double branched covers of knots for which we 
know the tunnel number or the transient number. 
For the case of tunnel number, note that
if $K$ has tunnel number $n$, then $K$ is contained in a genus $(n+1)$-handlebody $V$, such that its
complement is another genus $(n+1)$-handlebody $W$. By taking $\Sigma[K]$, $V$ and $W$ lift to genus
$(2n+1)$-handlebodies, that is, give a genus $2n+1$ Heegaard decomposition of $\Sigma[K]$. This shows
that $H_1(\Sigma[K])$ is an abelian group of rank at most $2n+1$.

The following lemma is a general result of coverings which we will use often. The proof is a standard argument, we omit it.

\begin{lemma} \label{conexo} Let $M$ be a given 3-manifold. Let $\Sigma$ be a double cover of $M$ with covering function $p : \Sigma \rightarrow M$; and let $C \subset M$. If $M$ is path connected and $p^{-1}(C)$ is connected 
then $\Sigma$ is connected.
\end{lemma}

The following theorem is our first important result of this section. We will see that if we know the transient number of a knot 
we can construct the double branched cover of this knot and from there calculate its first homology group.

\begin{theorem}\label{cubiertageneral} If $K$ is a knot in $S^3$ such that $tr(K) = n$, then the first homology group of the double branched cover of $K$ has a presentation with at most $2n + 1$ generators.
\end{theorem}

\begin{proof} Let $K$ be a knot in $S^3$ such that $tr(K) = n$, let $\{ \tau_1,\tau_2,...,\tau_n\}$ be a transient system for $K$, and let $T =\mathcal{N}(K\cup \tau_1 \cup \tau_2 \cup, \dots, \cup \tau_n)$, this is a genus $n+1$ handlebody. Let 
$K' \subset T$ be the trivial knot, such that $K'$ is homotopic to $K$ in $T$.

Let us define a family of compressing disks for $\partial T$ properly embedded in $T$, say 
$\{D_1, D_2, \dots , D_n, D_{n+1}\}$, which satisfy the following properties (see Figure 13):
\begin{enumerate}

\item For each $i \in \{1, 2, . . . , n\}$ the disk $D_i$ is properly embedded in $\mathcal{N}(\tau_i)$.
\item The disk $D_{n+1}$ is properly embedded in $\mathcal{N}(K)$ and is a compression disk for it.
\end{enumerate}

All of these disks are properly embedded in $T$, so we can deduce that:
\begin{enumerate}
\item The family $\{D_1, D_2, \dots, D_n, D_{n+1}\}$ is pairwise disjoint. 
\item For each $i \in \{1,2,\dots,n\}$, $\vert D_i \cup K \vert = 0$.
\item $\vert D_{n+1} \cap  K \vert = 1$.
\end{enumerate}

Let $\Sigma[K']$ be the double branched cover of $K'$ with covering function given by
$p : \Sigma[K'] \rightarrow S^3$. Note that $\Sigma[K']$ is homeomorphic to $S^3$.

\begin{claim}\label{disjointcurves} For each $i \in \{1, 2, \dots, n\}$, $p^{-1}(\partial D_i)$ has exactly two connected
components, where each connected component is a simple closed curve
in $p^{-1}(\partial T)$; whereas $p^{-1}(\partial D_{n+1})$ is a single simple closed curve in
$p^{-1}(\partial T)$. Also, all these curves are disjoint in $\partial T$. \end{claim}

\begin{proof} We know that $\vert D_{n+1} \cap K \vert = 1$ and $\vert D_i \cap K \vert = 0$ for all 
$i \in \{1,2,\dots,n\}$. As $K'$ is homotopic to $K$ in $T$, then $\vert D_{n+1}\cap K'\vert$ is an odd integer and 
$\vert D_i \cap K'\vert$ is an even integer for all $i \in \{1,2,\dots,n\}$. 
Therefore, for each $i \in \{1,2,...,n\}$ we have that $p^{-1}(\partial D_i)$ has exactly two connected
components in $p^{-1}(\partial T)$, where each connected component is a simple closed curve; and
$p^{-1}(\partial D_{n+1})$ is a simple closed connected curve in $p^{-1}(\partial T)$.
Now, since the disks of the family $\{D_1, D_2, \dots, D_{n+1}\}$ are pairwise disjoint, 
we have that all the curves are pairwise disjoint. \end{proof}

\begin{claim}\label{genus} $p^{-1}(\partial T)$ is a connected, orientable surface with Euler characteristic $-4n$ 
(and genus $2n + 1$) contained in $\Sigma[K']$. \end{claim}

\begin{proof} Note that $\partial T$ is a genus $n+1$ surface, then $\chi(\partial T) = -2n$, and therefore 
$\chi (p^{-1}(\partial T))) = 2\chi(\partial T ) = -4n$.
Since $\partial T$ is connected, $p^{-1}(\partial T)$ is a double cover of 
$\partial T$,  $\partial D_{n+1} \subset \partial T$ and $p^{-1}(\partial D_{n+1})$ is a connected curve on 
$p^{-1}(\partial T)$, then by Lemma \ref{conexo} we have that $p^{-1}(\partial T)$ is connected. Therefore 
$p^{-1}(\partial T)$ is a connected orientable surface of Euler characteristic $-4n$ (and of genus $2n + 1$).\end{proof}

\begin{claim}\label{otroconexo} $p^{-1}(\partial T \backslash \cup ^{n}_{j=1} \partial D_j)$ is connected. \end{claim}

\begin{proof} Clearly $\partial  T \backslash \cup^{n}_{j=1}\partial D_j$ is connected. We have that 
$p^{-1}(\partial T  \backslash \cup ^{n}_{j=1} \partial D_j)$ is a double cover of 
$\partial T \backslash \cup^{n}_{j=1}\partial D_j$, 
 that $\partial D_{n+1} \subset \partial T \backslash \cup ^{n}_{j=1} \partial D_j$ and that
$p^{-1}(\partial D_{n+1})$ is a connected curve on $p^{-1}(\partial T \backslash \cup ^{n}_{j=1} \partial D_j$), then using Lemma \ref{conexo} we have that $p^{-1}(\partial T \backslash \cup^{n}_{j=1}\partial D_j)$ is connected. \end{proof}

By Claim \ref{disjointcurves} we know that for each $i \in \{1,2,\dots,n\}$ the curve
$\partial D_i$ lifts, under $p$, to exactly two simple closed curves in $p^{-1}(\partial T)$. 
Let us denote by $\Lambda^{i}_1$
and $\Lambda^{i}_2$ the two liftings of $\partial D_i$ in $p^{-1}(\partial T)$, so 
$\{\Lambda^1_1,\Lambda^1_2,\Lambda^2_1,\Lambda^2_2,...,\Lambda^n_1,\Lambda^n_2\}$ is a pairwise disjoint 
collection of simple closed curves in $p^{-1}(\partial T)$. Also, $\Lambda^i_1 \cup \Lambda^i_2$ is a double cover of 
$\partial D_i$ with $p\vert_{\Lambda^i_1\cup \Lambda^i_2}$ the corresponding covering function, then the functions
$p\vert_{\Lambda^i_1}: \Lambda^i_1 \rightarrow \partial D_i$ and $p \vert_{\Lambda^i_2} : \Lambda^i_2 \rightarrow \partial D_i$ are homeomorphisms.

By Claim \ref{disjointcurves} we have that $p^{-1}(D_{n+1})$ is a simple closed curve on $p^{-1}(\partial T)$. 
Let us denote by $\Lambda$ the curve $p^{-1}(\partial D_{n+1})$. So $\Lambda$ is a double cover for $\partial D_{n+1}$ with covering function 
$p \vert_\Lambda : \Lambda \rightarrow \partial D_{n+1}$.

Let us introduce the following notations: 

\begin{itemize}

\item  $Ext(T):=S^3 \backslash T$,
\item $\Sigma [Ext(T)] := \Sigma [K'] \backslash p^{-1}(T)$,
\end{itemize}

Note that $\Sigma[Ext(T)]$ is a double cover of $Ext(T)$. 
Note also that $\partial \Sigma [Ext(T )] = p^{-1}(\partial T)$.

Let $\Sigma [Ext(K)]$ be the 3-manifold obtained from $\Sigma [Ext(T)]$ by adding a 2-handle
along each of the members of the family of curves 
$\{\Lambda^1_1, \Lambda^1_2, \Lambda^2_1, \Lambda^2_2, \dots, \Lambda^n_1, \Lambda^n_2\}$. Since the functions 
$p \vert_{\Lambda^i_r}$ are homeomorphisms for each $i \in \{1,2,\dots,n\}$ and $r \in \{1,2\}$, we can extend these 
homeomorphisms to homeomorphisms whose domains are discs whose boundaries are $\Lambda^i_r$, which map
to the disks $D_i$. We then extend these last homeomorphisms to homeomorphisms from the 2-handle added
along $\Lambda^i_r$ to $\mathcal{N}(D_i)$. With this we conclude that $\Sigma[Ext(K)]$ is a double 
cover of $Ext(T) \cup (\cup ^n_{j=1} \mathcal{N}(D_j))$. Recall that the family of disks $\{D_1, D_2, \dots, D_n\}$ was chosen 
such that $Ext(T ) \cup (\cup ^n_{j=1} \mathcal{N}(Dj ))$ is homeomorphic to $Ext(K)$. Therefore $\Sigma [Ext(K)]$ 
is a double cover of $Ext(K)$.

On the other hand, from Claim \ref{genus} we know that $p^{-1}(\partial T)$ is an orientable connected surface of 
genus $2n + 1$ and by Claim \ref{otroconexo} we know that $p^{-1}(\partial T\backslash \cup^n_{j=1} \partial D_j)$ is 
connected. Since $\{\Lambda^i_1,\Lambda^1_2,\Lambda^2_1,\Lambda^2_2,...,\Lambda^n_1,\Lambda^n_2\}$ 
consist of $2n$ curves and

$$p^{-1}(\partial T\backslash \cup ^n_{j=1} \partial D) = p^{-1}(\partial T)\backslash \cup_{i\in \{1,2,...,n\}  r\in \{1,2\}} \Lambda ^i_r ,$$

\noindent then $\partial\Sigma[Ext(K)]$ is an orientable surface of genus one.

Now, note that $\partial D_{n+1} \subset \partial Ext(K)$ since $\partial D_{n+1} \subset \partial \mathcal{N}(K)$ and 
$D_{n+1} \cap  D_i = \emptyset$ for all $i \in \{1,2,...,n\}$. Therefore we also have $\Lambda \subset \partial\Sigma[Ext(K)]$.

Let us define the 3-manifold $\Sigma[K]$ obtained from $\Sigma[Ext(K)]$ by adding a 2-handle along $\Lambda$ 
on $\partial \Sigma[Ext(K)]$, and then complete with a 3-ball so that $\Sigma[K]$ is a closed 3-manifold. 
Since $p\vert_\Lambda$ is a two-to-one covering function then we can extend this function to a function 
that goes from a disk, whose boundary is $\Lambda$, to the disk $D_{n+1}$, where this extension is two-to-one 
branched at the point $K \cap D_{n+1}$. 
This last function is then extended to a function that goes from the 2-handle added along
$\Lambda$ to 
$\mathcal{N}(D_{n+1})$, where this function is two to one branched along the arc $K \cap \mathcal{N}(D_{n+1})$. 
Finally, this last function is 
extended to the added 3-ball, thus obtaining a function that goes from $\Sigma[K]$ to $S^3$ which is two to one branched along the knot $K$. From the above we conclude that $\Sigma[K]$ is the double branched cover of $K$.

Now we know from Claim \ref{genus} that $p^{-1}(\partial T)$ is an orientable connected surface of genus $2n + 1$ contained in 
$S^3$. Since $\partial \Sigma[Ext(T )] = p^{-1} (\partial T)$ and $\Sigma[Ext(T )] \subset \Sigma[K'] = S^3$ then 
$H_1(\Sigma[Ext(T)])$ is a free abelian group of rank $2n+1$. So, let 
$H_1(\Sigma[Ext(T)])=<\theta_1,\theta_2,...,\theta_{2n+1}>$,
where $\theta_i$ for $i \in \{1,2,...,2n + 1\}$ are generators.

Thus, $H_1(\Sigma[K]) =< \theta_1,\theta_2,...,\theta_{2n+1} \, \vert \, \lambda_1^1,\lambda_1^2,\lambda_2^1, \lambda_2^2,...,\lambda_n^1,\lambda_n^2,\lambda >$,
where $\lambda$ and the $\lambda^j_r$, for $j \in \{1,2,...,n\}$ and $r \in \{1,2\}$, correspond to the homology classes in 
$H_1(\Sigma[Ext(T)])$ of the respective curves $ \Lambda$ and $\Lambda_r^j$.
\end{proof}

It should be noted that in the proof of Theorem \ref{cubiertageneral}, besides from proving
the result, we construct the double cover of $S^3$ branched along the knot for which we know the transient number.
This construction will continue to be repeated throughout this work. Theorem \ref{cubiertageneral} can be generalized to  $p$-fold branched covers, with a similar proof.

\begin{theorem}\label{cubiertageneral-p} If $K$ is a knot in $S^3$ such that $tr(K) = n$, then the first homology group of the $p$-fold branched cover of  $K$ has a presentation with at most $pn + 1$ generators.
\end{theorem}

The next lemma is a general result of algebra of groups, which we will use for the proof of Theorems \ref{cubiertatunel} 
and \ref{cubiertatransito}.

\begin{lemma}\label{grupos} Let $G_1$ and $G_2$ be abelian groups such that

$G_1 =< \theta_1,\theta_2,\theta_3 : \lambda_1,\lambda_2,\lambda_3 >$ and $G_2 =< \beta_1,\beta_2 : \delta_1,\delta_2 >$.

Let $\Psi :< \theta_1,\theta_2,\theta_3 >\rightarrow < \theta_1,\theta_2,\theta_3 >$ and 
$\Phi :< \theta_1,\theta_2,\theta_3 >\rightarrow < \beta_1,\beta_2 >$ be homomorphisms between free abelian groups such that:

$$\begin{array}{ccccc}
\Psi(\theta_1) = \theta_2 & \Psi(\lambda_1) = \lambda_2 & \vert & \Phi(\theta_1) = \beta_1 & \Phi(\lambda_1) = \delta_1 \\
 \Psi(\theta_2) = \theta_1 & \Psi(\lambda_2) = \lambda_1 & \vert & \Phi(\theta_2) = \beta_1 & \Phi(\lambda_2) = \delta_1 \\
 \Psi(\theta_3) = \theta_3  &  \Psi(\lambda_3) = \lambda_3 & \vert &  \Phi(\theta_3) = 2\beta_2 & \Phi(\lambda_3) = 2\delta_2
 \end{array}$$

If $\lambda_1 = x\theta_1 + y\theta_2 + z\theta_3$ and $G_2$ is the trivial group, then $G_1$ is isomorphic to $Z_{x-y}$.
\end{lemma}

\begin{proof} Let $a_{ij}$ be integers, with $i, j \in \{1, 2, 3\}$, such that: 
\begin{equation} \label{sistem1}
\begin{split}
\lambda_1 = a_{11}\theta_1 + a_{12}\theta_2 + a_{13}\theta_3 \\
\lambda_2 = a_{21}\theta_1 + a_{22}\theta_2 + a_{23}\theta_3  \\
\lambda_3 = a_{31}\theta_1 + a_{32}\theta_2 + a_{33}\theta_3
\end{split}
\end{equation}

Applying the homomorphism $\Psi$, on both sides of the previous system of equations, we obtain:

\begin{equation} \label{sistem2}
\begin{split}
\lambda_2 = \Psi(\lambda_1) = \Psi(a_{11}\theta_1 + a_{12}\theta_2 + a_{13}\theta_3) = a_{11}\theta_2 + a_{12}\theta_1 + a_{13}\theta_3 \\
\lambda_1 = \Psi(\lambda_2) = \Psi(a_{21}\theta_1 + a_{22}\theta_2 + a_{23}\theta_3) = a_{21}\theta_2 + a_{22}\theta_1 + a_{23}\theta_3 \\
\lambda_3 = \Psi(\lambda_3) = \Psi(a_{31}\theta_1 + a_{32}\theta_2 + a_{33}\theta_3) = a_{31}\theta_2 + a_{32}\theta_1 + a_{33}\theta_3
\end{split}
\end{equation}

From the system (\ref{sistem1}) and from the system obtained in (\ref{sistem2}) we get:

\begin{equation} \label{sistem3}
\begin{split}
0 = (a_{11} - a_{22})\theta_1 + (a_{12} - a_{21})\theta_2 + (a_{13} - a_{23})\theta_3 \\
0 = (a_{12} - a_{21})\theta_1 + (a_{11} - a_{22})\theta_2 + (a_{13} - a_{23})\theta_3  \\
0 = (a_{31} - a_{32})\theta_1 + (a_{32} - a_{31})\theta_2
\end{split}
\end{equation}

Since $< \theta_1,\theta_2,\theta_3 >$ is a free abelian group, then from the system in (\ref{sistem3}) we have:

$a_{11} = a_{22}, \quad a_{12} = a_{21}, \quad a_{13} = a_{23}, \quad a_{31} = a_{32}$

Then the system (\ref{sistem1}) can be rewritten as

\begin{equation} \label{sistem4}
\begin{split}
\lambda_1 = a_1\theta_1 + a_2\theta_2 + a_3\theta_3 \\
\lambda_2 = a_2\theta_1 + a_1\theta_2 + a_3\theta_3  \\
\lambda_3 = a_4\theta_1 + a_4\theta_2 + a_5\theta_3 \\
\end{split}
\end{equation}

\noindent where $a_1=a_{11}$, $a_2=a_{12}$, $a_3=a_{23}$, $a_4=a_{31}$ and $a_5=a_{33}$.
Applying the homomorphism $\Phi$ to the system (\ref{sistem4}) we obtain:

\begin{equation} \label{sistem5}
\begin{split}
\delta_1 =\Phi(\lambda_1)=\Phi(a_1\theta_1 +a_2\theta_2 +a_3\theta_3) = (a_1 +a_2)\beta_1 +2a_3\beta_2     \\
\delta_1 = \Phi(\lambda_2) = \Phi(a_2\theta_1 + a_1\theta_2 + a_3\theta_3) = (a_2 + a_1)\beta_1 + 2a_3\beta_2  \\ 
2\delta_2 = \Phi(\lambda_3) = \Phi(a_4\theta_1 + a_4\theta_2 + a_5\theta_3) = 2a_4\beta_1 + 2a_5\beta_2
\end{split}
\end{equation}

By properties of free abelian groups, we obtain from the last equation of the system (\ref{sistem5}) that:
$$\delta_2 = a_4\beta_1 + a_5\beta_2$$. 

So the system in (\ref{sistem5}) can be rewritten as:

\begin{equation} \label{sistem6}
\begin{split}
\delta_1 = (a_1 + a_2)\beta_1 + 2a_3\beta_2.  \\ 
\delta_2 = a_4\beta_1 + a_5\beta_2
\end{split}
\end{equation}

From the system (\ref{sistem6}) we see that the matrix $A$, given by:

$$A = \begin{pmatrix} a_1+a_2  & 2a_3 \\ a_4  & a_5 \end{pmatrix}$$
is the representation matrix of the group $ G_2 = < \beta_1, \beta_2 : \delta_1, \delta_2 >$. From the system
in (\ref{sistem4}), doing an operation on rows, we see that the matrix $\tilde A$, given by:

$$\tilde A= \begin{pmatrix} a_1 & a_2 & a_3 \\ a_1+a_2 & a_1+a_2 & 2a_3 \\ a_4 & a_4 & a_5 \end{pmatrix}$$
is a representation matrix of the group $G_1$.

By Smith Normal Form Theorem, there exists matrices $S_1$ and $S_2$ of order $2 \times 2$, invertible and with integer entries such 
that the matrix $S_1AS_2$ is a diagonal matrix with integer entries. From Smith Normal Form Theorem it is also known that the 
inverse matrices of $S_1$ and $S_2$ have integer entries, therefore $det S_1 = \pm 1$ and $det S_2 = \pm 1$. Now, since $G_2$ is the trivial 
group, then $det A = \pm 1$. So the matrix $S_1AS_2$ is of the form 

\begin{equation}\label{sistem7} 
S_1AS_2 = \begin{pmatrix} \pm 1 & 0 \\ 0 & \pm 1 \end{pmatrix}
\end{equation}

From (\ref{sistem7}) we can ensure that there is a matrix $S$ of order $2 \times 2$, invertible and
with integer entries that satisfies:

\begin{equation}\label{sistem8} 
SA= \begin{pmatrix} 1 & 0 \\ 0 & 1 \end{pmatrix}
\end{equation}

Let us define the following matrix:

\begin{equation*}
\tilde{S}=
\begin{pmatrix}
1 & \begin{matrix}
0 & 0
\end{matrix}\\
\begin{matrix}
0 \\
0
\end{matrix} & S
\end{pmatrix}
\end{equation*}

Clearly the matrix $\tilde S$ has integer entries and using the result in (\ref{sistem8}) we have:

\begin{equation}\label{sistem9} 
\tilde S \tilde A= \begin{pmatrix} a_1 & a_2 &  a_3 \\ 1 & 1 &  0 \\ 0 & 0 & 1 \end{pmatrix}
\end{equation}

Using elementary operations, from the matrix in (\ref{sistem9}) we obtain:

$$\begin{pmatrix} a_1-a_2 & 0 & 0 \\ 0 & 1 &  0 \\ 0 & 0 & 1 \end{pmatrix}$$

From the above matrix we conclude that the group $< \theta_1, \theta_2, \theta_3 : \lambda_1, \lambda_2, \lambda_3 >$ is
isomorphic to $Z_{a_1-a_2}$, therefore the group $G_1$ is isomorphic to $Z_{a_1-a_2}$. \end{proof}

The following result is well known to experts. We include a proof for completeness and because it
will help us as a lemma in the proof of Theorem \ref{cubiertatransito}.

\begin{theorem} \label{cubiertatunel} If $K$ is a knot in $S^3$ such that $t(K) = 1$, then the first homology group of the double branched cover of $K$ is cyclic. \end{theorem}

\begin{proof} Let $K$ be a knot in $S^3$ such that $t(K) = 1$, and let ${\tau}$ be an unknotting tunnel for $K$. 
Let $T = \mathcal{N}(K \cup \tau)$ and 
$Ext(T ) = S^3 \backslash T$, so $Ext(T)$ is a genus two handlebody.
Since $Ext(T)$ is a handlebody, we can ensure that there exists a knot $K' \subset T$ such that $K'$ is a trivial knot in 
$S^3$ and it is homotopic with the knot $K$ in $T$. Let $\Sigma(K')$ be the double branched cover of the 
knot $K'$ and let $p : \Sigma(K') \rightarrow S^3$ be the associated covering function.
It is easy to notice, for the way it is defined $T$,  that there are meridian disks $D_1$ and $D_2$ in $T$ such that 
$\vert D_1 \cap K\vert = 0$ and 
$\vert D_2 \cap K\vert = 1$. Since $K'$ is homotopic to $K$ in $T$, then $\vert D_1 \cap K' \vert$ is an even integer and 
$\vert D_2 \cap K'\vert$ is an odd integer. Therefore $\partial D_1$ lifts, under $p$, in two simple closed curves; while 
$\partial D_2$ lifts to exactly a single simple closed curve. Let us denote by $\Lambda_1$ and 
$\Lambda_2$ the liftings of $\partial D_1$ and by $\Lambda_3$ the lifting of $\partial D_2$.
For each $i \in \{1,2,3\}$ we attach a 2-handle to $p^{-1}(Ext(T))$ along
$\Lambda_i \subset \partial(p^{-1}(Ext(T)))$; let us denote the 2-handle attached along $\Lambda_i$ by 
$\overline{\Lambda}_i$. Let $\Sigma$ be the 3-manifold obtained by attaching to $p^{-1}(Ext(T))$ the 2-handles 
$\overline{\Lambda}_i$, 
that is: $\Sigma := p^{-1}(Ext(T))\cup (\cup_{i=1}^3\overline{\Lambda}_i)$.

Let us note the following observations: 
\begin{enumerate}

\item $\partial p^{-1}(Ext(T))$ is a genus three connected surface.
\item $p^{-1}(Ext(T))$ is a double covering of $Ext(T)$. 
\item The function $p$ can be extended to $\Sigma$, such that $\overline{\Lambda}_1 \cup \overline{\Lambda}_2$ is a double 
covering of $\mathcal{N}(D_1)$ and $\Lambda_3$ is a double covering of $\mathcal{N}(D_2)$ branched along 
$K \cap \mathcal{N}(D_2)$.
\item $\partial\Sigma$ is a 2-sphere.

\end{enumerate}

Let $\Sigma(K)$ be the 3-manifold obtained by attaching a 3-ball to $\Sigma$ along its
boundary. So, we can extend the covering function $p\vert_ {p^{-1}(Ext(T))} : p^{-1}(Ext(T)) \rightarrow Ext(T)$ to a covering 
function $p': \Sigma(K) \rightarrow S^3$ which branches along the knot $K$. 
Therefore $\Sigma (K)$ is the double covering of 
$S^3$ branched along $K$ with covering function given by $p'$.

We know that $Ext(T)$ is a genus two handlebody, therefore $H_1(Ext(T))$ is a free abelian group in two generators. 
Note that $\partial(p^{-1}(Ext(T)))$ is a genus three handlebody, therefore $H_1(p^{-1}(Ext(T)))$ is a free abelian group in three generators.

\begin{claim}\label{curvascubiertas} There are two connected simple closed curves in $Ext(T)$, 
denoted by $B_1$ and $B_2$, 
such that $B_1$ lifts, by $p$, in two  closed and connected simple curves, denoted by $\Theta_1$ 
and $\Theta_2$; while $B_2$ lifts, by $p$, in exactly one simple curve closed, denoted 
by $\Theta_3$. If $\beta_j$ is the homology class of $B_j$ in $H_1(Ext(T))$ and $\theta_i$ is the homology 
class of $\Theta_i$ in $H_1(p^{-1}(Ext(T)))$ for all $j \in \{1,2\}$ and $i \in \{1,2,3\}$, then
$H_1(Ext(T))=<\beta_1,\beta_2 >$ , $H_1(p^{-1}(Ext(T)))=<\theta_1,\theta_2,\theta_3 >$.
 \end{claim}

\begin{proof} Note that $Ext(T)$ is a genus two handlebody, call it $V$. Let $D$ be a disk in $V$ which splits it
in two solid tori $V_1$ and $V_2$. Note that $p^{-1}(V_i)$ double covers $V_i$, then it is either a set of
two solid tori or a solid torus that coves $V_i$ two-to-one. There are two possibilities.

\begin{enumerate} 

\item $V_1$ is covered by two solid tori, say $V^1_1$ and $V^2_1$, and $V_2$ is covered two-to-one by a solid torus $V_2'$. See Figure

\item $V_1$ and $V_2$ are covered both two-to-one  by solid tori $V_1'$ and $V_2'$. See Figure

\end{enumerate}

In Case 1, take as $B_i$, $i=1,2$, a core of the solid tori $V_i$. Clearly $B_1$ lifts to two simple closed curves
$\Theta_1$ and $\Theta_2$, which are a core of the solid tori $V^1_1$ and $V^2_1$, and $B_2$ lifts to
a simple closed curve $\Theta _3$ which is a core of the solid tori $V_2'$, and which cover two-to-one the
curve $B_2$. In this case it is clear that the homology classes of the curves satisfy the required properties.
See Figure \ref{handlebody1}.

\begin{figure}

\includegraphics[angle=0, width=8true cm]{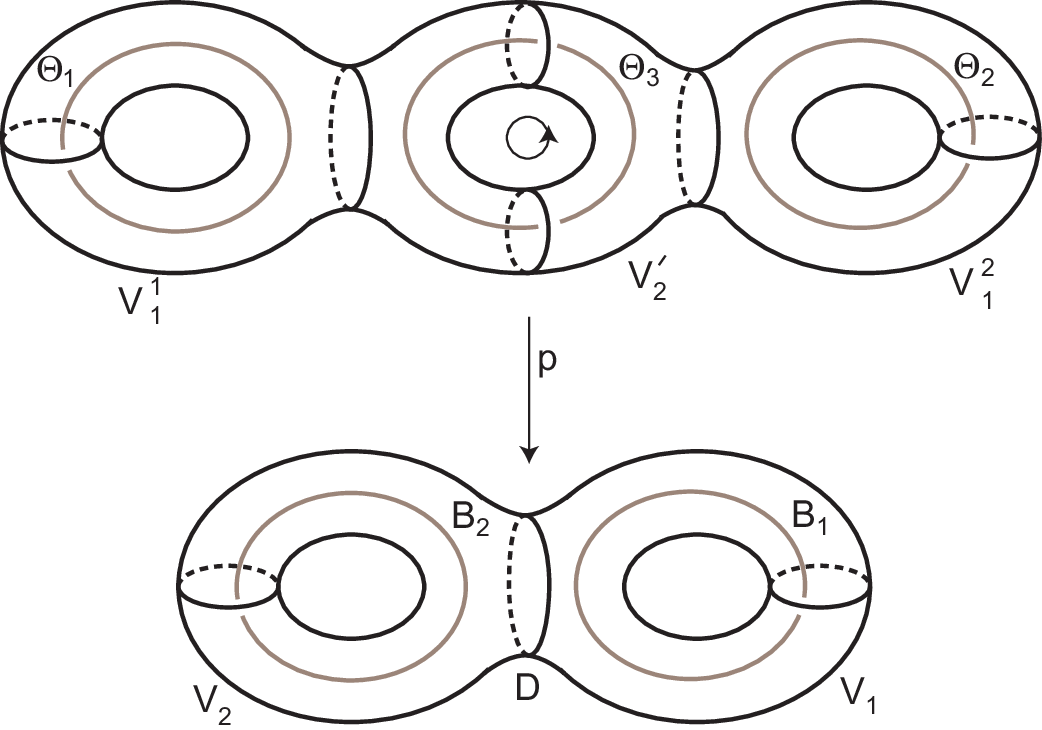}

\caption{}
\label{handlebody1}

\end{figure}

In Case 2, take as $B_1$ a curve that goes once around each of the cores of $V_1$ and $V_2$ and intersects
$D$ in two points. In this case $B_1$ lifts to two simple closed curves $\Theta_1$ and $\Theta_2$, each of
which goes once around $V^1_1$ and $V^2_1$. Take as $B_2$ a core of $V_1$, then clearly it lifts to a 
curve $\Theta_3$ which covers $B_2$ two-to-one. It is clear that the homology classes of the curves satisfy the required properties. See Figure \ref{handlebody2}.
\end{proof}

\begin{figure}

\includegraphics[angle=0, width=8true cm]{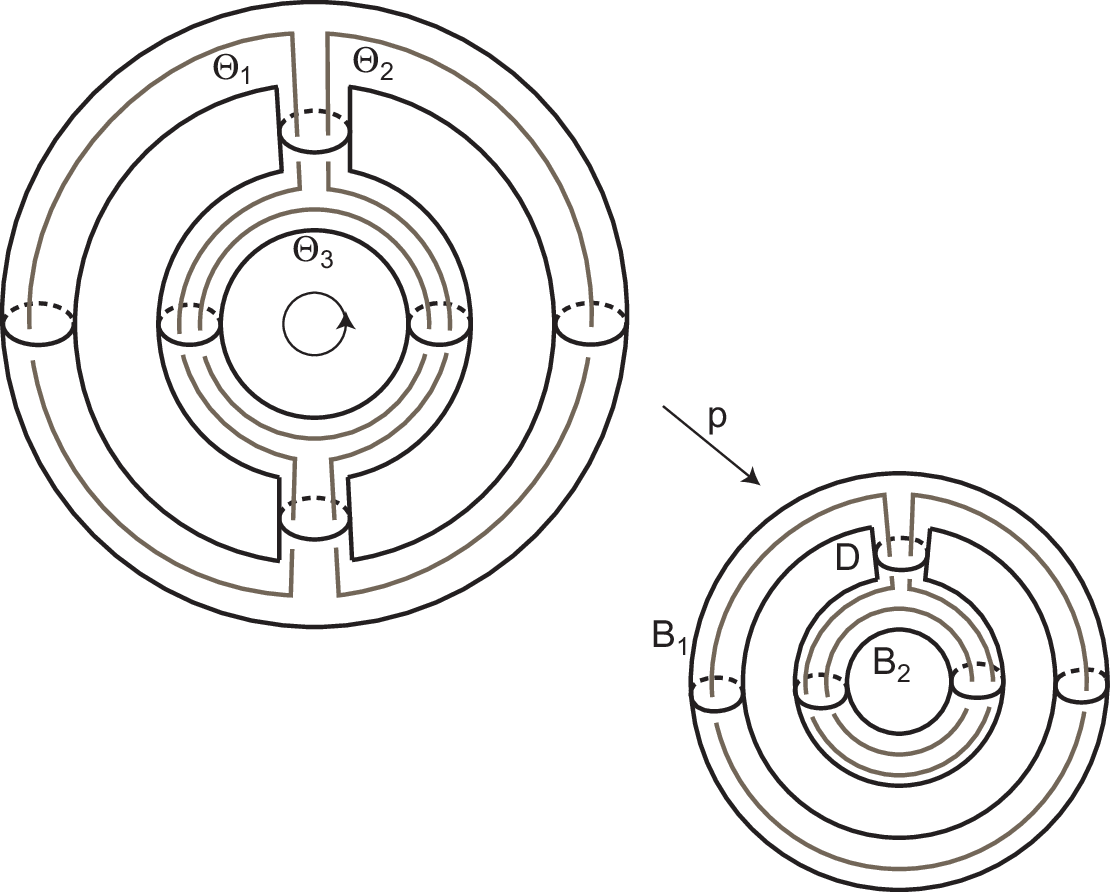}

\caption{}
\label{handlebody2}

\end{figure}

We know that $p^{-1}(Ext(T))$ is a double covering of $Ext(T)$, with covering function given by the restriction of $p$. Let 
$p_*: H_1(p^{-1}(Ext(T))) \rightarrow H_1(Ext(T))$ be the
homomorphism associated with the restriction of $p$. For each $i \in \{1, 2, 3\}$, let us denote
by $\lambda_i$ the homology class in $H_1(p^{-1}(Ext(T)))$ associated to the curve $\Lambda_i$. 
Note that $H_1(\Sigma[K])=<\theta_1,\theta_2,\theta_3: \lambda_1,\lambda_2,\lambda_3>$. For
each $j \in \{1, 2\}$, let us denote by $\delta_j$ the homology class in $H_1(Ext(T))$ associated to the curve $\partial D_j$.
We have that $H_1(Ext(T))=<\beta_1,\beta_2: \delta_1, \delta_2 >$.
By choosing orientations conveniently, assume that 

\begin{equation} \label{sistem10}
\begin{split}
 p_*(\lambda_1) = \delta_1, \,\,\, p_*(\lambda_2) = \delta_1, \,\,\, p_*(\lambda_3) = 2\delta_2
\end{split}
\end{equation}

According to Claim \ref{curvascubiertas}, we have that

\begin{equation} \label{sistem11}
\begin{split}
p_*(\theta_1)= \beta_1, \,\,\, p_*(\theta_2)=\beta_1, \,\,\, p_*(\theta_3)=2\beta_2
\end{split}
\end{equation}

Let $q:p^{-1}(Ext(T)) \rightarrow p^{-1}(Ext(T))$ be the non-trivial covering transformation associated to the covering function $p\vert_{p^{-1}(Ext(T ))}$. Let 
$q_*:H_1(p^{-1}(Ext(T))) \rightarrow H_1(p^{-1}(Ext(T)))$ be  the homomorphism induced by the covering transformation $q$. By Claim \ref{curvascubiertas} we have that

\begin{equation} \label{sistem12}
\begin{split}
q_*(\theta_1) = \theta_2, \,\,\, q_*(\theta_2) = \theta_1, \,\,\, q_*(\theta_3) = \theta_3 \\
q_*(\lambda_1) = \lambda_2, \,\,\, q_*(\lambda_2) = \lambda_1, \,\,\, q_*(\lambda_3) = \lambda_3
\end{split}
\end{equation}

Then, applying Lemma \ref{grupos} directly we have that $H_1(\Sigma_2(K)) = Z_{x-y}$,
where $\lambda_1 = x\theta_1 + y\theta_2 + z\theta_3$.
\end{proof}

Now we prove the main result of this paper.

\begin{theorem} \label{cubiertatransito} If $K$ is a knot in $S^3$ such that $tr(K) = 1$ then the first homology group of the double branched cover of $K$ is cyclic. \end{theorem}

\begin{proof} Let $K$ be a knot in $S^3$ such that $tr(K) = 1$, and let $\{\tau \}$ be a transient system for the 
knot $K$. Let $T = \mathcal{N}(K \cup \tau)$ and let $K' \subset T$ be a trivial knot in $S^3$ such that $K'$ is 
homotopic to $K$ in $T$. Define also the 3-manifold $Ext(T)$ as $Ext(T) := S^3 \backslash Int(T)$.

As $\partial T$ is a genus two surface in the exterior of the knot $K'$, which is trivial, it follows that $\partial T$
is compressible in $Ext(K')$, that is, there is a compression disk $E_1$ for $\partial T$ disjoint from $K'$.

There are two possibilities for the disk $E_1$, that is: 
\begin{enumerate}

\item The disk $E_1$ is a compression disk for $\partial T$ lying in the interior of $T$; 
\item The disk $E_1$ is a compression disk for $\partial T$ lying in the exterior of $T$.
\end{enumerate}

Suppose first that we have case (1), that is, $E_1$ lies in the interior of $T$. If $E_1$ separates $T$, then by
cutting along $E_1$ we get two solid tori, one of them contains $K'$, and then there is a compression disk in the
other solid tori which is non-separating in $T$. So, we can assume that there is a compression disk $E_1$
for $\partial T$, lying in $T$, and which does not separate $T$.

\begin{claim} There exist a knot $K''$ and a disk $E_2$ in $T$ such that:

\begin{enumerate}
\item $E_2$ is a compression disk for $\partial T$ which is properly embedded in $T$.
\item $K''$ is a trivial knot in $S^3$ and it is homotopic to $K$ in $T$.

\item $\vert E_2 \cap K'' \vert = 1$.
\end{enumerate} \end{claim}

\begin{proof} By cutting $T$ along $E_1$, we get a solid torus $V$. The knot $K'$ lies in $V$, and as $K'$
represents a primitive element in $\pi_1(T)$, it must be 
homotopic to the core of $V$.
If $V$ is knotted, then $\partial V$ is incompressible in $Ext(K')$, which is not possible, for $K'$ is the trivial knot.
Then $V$ must be a standard solid torus in $S^3$. Then $K'$ can be further homotoped to the core of $V$,
which is a trivial knot in the 3-sphere. Then there is a disk $E_2$ in $V$ such that $\vert E_2 \cap K'' \vert = 1$.
\end{proof}

Let $\Sigma[K'']$ be the double cover of $S^3$ branched along $K''$ with covering function given by $p:\Sigma[K''] \rightarrow S^3$. The disk $E_1$ and $E_2$ form a meridian disk system for $T$, and as $K''$ is disjoint from $E_1$ and 
intersects $E_2$ in one point, it follows that $p^{-1}(T)$ is a genus 3 handlebody, $p^{-1}(E_1)$ consists of two disks and 
$p^{-1}(E_2)$ consists of a single disk which covers two-to-one the disk $E_2$. Note that these disks form a
 meridian system for $p^{-1}(T)$. Let $B_i=\partial E_i$, $i=1,2$. Denote by $\Theta_1$ and $\Theta_2$ the two
 components of $p^{-1}{B_1}$, and let $\Theta_3=p^{-1}(B_2)$. As $\Sigma[K'']$ is the 3-sphere, and $p^{-1}(T)$
 is a genus 3 handlebody, it follows that the homology clases of the curves $\Theta_i$, $i=1,2,3$, generate
 $H_1(p^{-1}(Ext(T))$.

Let $\{D_1,D_2\}$ compression disks in the interior of $T$ such that $D_1$ is properly embedded in $\mathcal{N}(\tau)$ y $D_2$ is properly embedded in $\mathcal{N}(K)$, such that $\vert D_1 \cap K \vert = 0$ and 
$\vert D_2 \cap K \vert=1$. Note that the disks
$D_1$ and  $D_2$ do not separate $T$. As $K''$ is homotopic to $K$ in $T$, then $\vert D_1 \cap K'' \vert$ is an even number and $\vert D_2 \cap K'' \vert$ is an odd number. Therefore $\partial D_1$ lifts, under $p$, in two simple 
closed curves, while  $\partial D_2$ lifts exactly in a single simple closed curve. Denote by 
$\Lambda_1$ y $\Lambda_2$ the liftings of $\partial D_1$
and by $\Lambda_3$ the lifting of $\partial D_2$.
Attach 2-handles to the 3-manifold $p^{-1}(Ext(T))$ along the curves $\Lambda_i$, note that these curves lie in 
$\partial (p^{-1}(Ext(T)))$, and denote the 2-handle attached along $\Lambda_i$ by $\overline \Lambda_i$. Let $\Sigma$ be the 3-manifold obtained by attaching to $p^{-1}(Ext(T))$ the 2-handles $\overline \Lambda_i$.


Note that $p^{-1}(Ext(T))$ is a doble covering of $Ext(T)$, with covering function $p'$ given by 
$p'=p\vert_{p^{-1} (Ext(T ))}$. The function $p'$ can be extended to a function 
$p':\Sigma \rightarrow Ext(T) \cup N(D_1) \cup N(D_2)$, such that $\overline \Lambda_1 \cup \overline \Lambda_2$ is a double covering of $\mathcal{N}(D_1)$ y $\overline \Lambda_3$ is a double covering of $\mathcal{N}(D_2)$ 
branched along $K \cap \mathcal{N}(D_2)$.

Note that $\partial \Sigma$ is a 2-sphere. Let $\Sigma(K)$ be the 3-manifold obtained by attaching a 3-ball to $\Sigma$
along its boundary. We can extend the covering function $p'$ to a covering function 
$\hat p: \Sigma(K) \rightarrow S^3$, which branchs along $K$. Therefore $\Sigma(K)$ is the double cover of $S^3$
branched along $K$ with covering function given by $\hat p$.

As $p^{-1}(Ext(T))$ is a double covering of $Ext(T)$, with covering function given by the restriction of $p$, 
let $p_*: H_1(p^{-1}(Ext(T))) \rightarrow H_1(Ext(T))$ be the homomorphism induced by $p$. For each $i \in \{1,2,3\}$
denote by $\lambda_i$ the homology class in $H_1(p^{-1}(Ext(T)))$ associated to the curve $\Lambda_i$.
For each $j \in \{ 1,2\}$ denote by $\delta_j$ the homology class in $H_1(Ext(T))$ associated to the curve $\partial D_j$. 
Then

\begin{equation} \label{sistem13}
\begin{split}
p_*(\lambda_1) = \delta_1, \,\,\, p_*(\lambda_2) = \delta_1, \,\,\,p_*(\lambda_3) = 2\, \delta_2 
\end{split}
\end{equation}

Note that $H_1(Ext(T))$ is a free abelian group in two generators, generated by the homology classes of 
the curves $B_1$ and $B_2$, which we denote by $\beta_1$ and $\beta_2$. As we said before, 
$H_1(p^{-1}(Ext(T)))$ is a free abelian group in there generators, generated by the homology classes of the
curves $\Theta_i$, which we denote by $\theta_i$, $i=1,2,3$. We have that

\begin{equation} \label{sistem14}
\begin{split}
H_1(Ext(T))=<\beta_1,\beta_2 >, \,\, \, H_1(p^{-1}(Ext(T)))=<\theta_1,\theta_2,\theta_3 >
\end{split}
\end{equation}

We also obtain that 

\begin{equation} \label{sistem15}
\begin{split}
H_1(Ext(T))=<\beta_1,\beta_2: \delta_1, \delta_2 >, \,\, \, H_1(\Sigma[K])=<\theta_1,\theta_2,\theta_3; \lambda_1, \lambda_2, \lambda_3 >
\end{split}
\end{equation}

\begin{equation} \label{sistem16}
\begin{split}
p_*(\theta_1) = \beta_1, \,\,\, p_*(\theta_2) = \beta_1, \,\,\, p(\theta_3) = 2\beta_2 
\end{split}
\end{equation}

Let $q: p^{-1}(Ext(T)) \rightarrow p^{-1}(Ext(T))$ be the non-trivial covering transformation, associated to the covering function $p$. Let $q_* : H_1(p^{-1}(Ext(T))) \rightarrow H_1(p{-1}(Ext(T)))$ be the homomorphism associated to the covering transformation $q$. By the way that $\theta_i$ and the $\lambda_i$ were defined we have that:

\begin{equation} \label{sistem17}
\begin{split}
q_*(\theta_1) = \theta_2, \,\,\, q_*(\theta_2) = \theta_1, \,\,\, q_*(\theta_3) = \theta_3, \\
 q_*(\lambda_1) = \lambda_2, \,\,\, q_*(\lambda_2) = \lambda_1, \,\,\, q_*(\lambda_3) = \lambda_3
\end{split}
\end{equation}
Applying Lemma \ref{grupos} we have that $H_1(\Sigma(K)) = \mathbb{Z}_{x-y}$, where $\lambda_1 = x\theta_1 + y\theta_2 + z\theta_3$.
So, we have proved that if the compression disk $E_1$ is contained in $T$, then the homology group of 
the double branched cover of $K$ is cyclic.

Now suppose that the compression disk $E_1$ is contained in $Ext(T)$. In this situation we can suppose that $Ext(T)$
is not a handlebody, for otherwise we have that $t(K) = 1$ and by Theorem \ref{cubiertatunel} we get the desired result. Suppose first that
the disk $E_1$ does not separate $Ext(T)$. Define $\Gamma= T \cup \mathcal{N}(E_1)$. As $E_1$ does not
 divide $\partial T$ then $\partial \Gamma$ is a connected genus one surface, and it must bound a solid torus.
 Then $\Gamma$ is a solid torus, for otherwise $Ext(T)$ will be a genus 2 handlebody. So, $\Gamma$ is a knotted
 solid torus and $K'$ lies on it. As $K'$ is a trivial knot, it must lie in a 3-ball contained in $\Gamma$, for otherwise there will be an incompressible torus in $Ext(K)$. In particular, $K'$ has winding number zero in $\Gamma$.
 Then $K$ is also of winding number zero in $\Gamma$, as it is homotopic to $K'$ in $T\subset \Gamma$. 
 Embed $\Gamma$ in $S^3$ such that it is an standard solid torus $V$, and such that a preferred longitude
 of $\Gamma$ goes to a preferred longitude of $V$. Let $\bar K$ be the image of $K$ in $V$. Then $K$ is
 a satellite knot with pattern given by $\bar K$. As $\bar K$ has winding number zero in $V$, it follows that
 $H_1(\Sigma[\bar K])$ is isomorphic to  $H_1(\Sigma[K])$, by \cite{Sei}.  Let $\bar T$ be the image of $T$ in $V$,
 clearly $\bar T$ is the neighborhood of $\bar K$ union a transient arc, and the exterior of $\bar T$ is the exterior
 of $V$, which is a solid torus union a 1-handle given by the image of the disk $E_1$.
This shows $\bar K$ is a tunnel number one knot
 and then $H_1(\Sigma[\bar K])$ is a cyclic group, which implies then that $H_1(\Sigma[K])$ is also cyclic.
 
 Suppose now that the disk $E_1$ separates $Ext(T)$ and that there is no non-separating compression disk in $Ext(T)$.
Let $\Gamma= T \cup \mathcal{N}(E_1)$. As $E_1$ is separating, $\partial \Gamma$ consist of two tori, 
say $S_1$ and $S_2$. Then $S_1$ bounds a solid torus $V_1$ which contains $\Gamma$, and also contains $S_2$.
Then $V_1$ is a knotted solid torus, and as $K'$ is contained in $V_1$, it must lie inside a 3-ball, and then as
in the previous case, $K$ has winding number zero in $V_1$. Embed $V_1$ in $S^3$ such that it is an standard solid 
torus $V_2$, and such that a preferred longitude of $V_1$ goes to a preferred longitude of $V_2$. Let $\bar K$ be the
image of $K$ in $V_2$. Then $K$ is
 a satellite knot with pattern given by $\bar K$. As $\bar K$ has winding number zero in $V$, it follows that
 $H_1(\Sigma[\bar K])$ is isomorphic to  $H_1(\Sigma[K])$, by \cite{Sei}.  Let $\bar T$ be the image of $T$ in $V$,
 clearly $\bar T$ is the neighborhood of $\bar K$ union a transient arc, and the exterior of $\bar T$ is the exterior
 of $V$, which is a solid torus union a manifold bounded by the image of $S_2$ plus 1-handle given by the image 
 of the disk $E_1$. It follows that $\bar K$ is a transient number one knot
 such that the exterior of the knot union a transient arc is compressible, and it has a non-separating compression disk.
 By the previous case, $H_1(\Sigma[\bar K])$ is a cyclic group, which implies then that $H_1(\Sigma[K])$ is also cyclic.
 \end{proof}
 

\section{Knots with large transent number}\label{examples}

By the results of the last section we can now estimate the transient number of some knots.

\begin{theorem}\label{largetransitnumber} Let $K$ be a knot such that its double branched cover is not an homology sphere, that is, $H_1(\Sigma[K])$ is not trivial. Then
\begin{enumerate}
\item $tr(K\# K) \geq 2$:
\item $tr(K_n) \geq (n-1)/2$, where $K_n = K\# K\# \cdots \# K$, is the connected sum of $n$ copies of $K$.
\end{enumerate}
\end{theorem}

\begin{proof} it is known that $\Sigma[K_n]= \Sigma[K] \# \Sigma[K] \# \cdots \# \Sigma[K]$, the connected sum
of $n$ copies of $\Sigma[K]$. As $H_1(\Sigma[K])$ is not trivial, then $H_1(\Sigma[K])$ has rank at least
$n$. By Theorem \ref{cubiertageneral}, $tr(K_n) \geq (n-1)/2$, this shows (2). In particular 
$H_1(\Sigma[K_2])=H_1(\Sigma[K])+H_1(\Sigma[K])$, which is not cyclic, and this implies (1).
\end{proof}

This shows that there are knots with arbitrarily large transient number, which answers a question of 
Koda and Ozawa \cite{KO}.

Now we concentrate in the tables of knots up to crossing number 10.

\begin{theorem} 
\begin{enumerate}
\item The following knots have transient number $2$: $8_{18}$, $9_{35}$, $9_{37}$, $9_{40}$, $9_{41}$, $9_{46}$, $9_{47}$,
$9_{48}$, $9_{49}$, $10_{74}$, $10_{75}$, $10_{98}$, $10_{99}$, $10_{103}$, $10_{123}$, $10_{155}$, $10_{157}$.

\item The following knots have transient number at most $2$: $8_{16}$, $9_{29}$, $9_{32}$, $9_{38}$, $10_{61}$,
$10_{62}$, $10_{63}$, $10_{64}$, $10_{65}$, $10_{66}$, $10_{67}$, $10_{68}$, $10_{69}$, 
$10_{79}$, $10_{80}$, $10_{81}$, $10_{83}$, $10_{85}$, $10_{86}$, $10_{87}$, $10_{89}$,
$10_{90}$, $10_{92}$, $10_{93}$, $10_{94}$, $10_{96}$, $10_{97}$, $10_{100}$, $10_{101}$,
$10_{105}$, $10_{106}$, $10_{108}$, $10_{109}$, $10_{110}$, $10_{111}$, $10_{112}$, 
$10_{115}$, $10_{116}$, $10_{117}$, $10_{120}$, $10_{121}$, $10_{122}$, $10_{140}$, 
$10_{142}$, $10_{144}$, $10_{148}$, $10_{149}$, $10_{150}$, $10_{151}$, $10_{152}$, 
$10_{153}$, $10_{154}$, $10_{158}$, $10_{160}$, $10_{162}$, $10_{163}$, $10_{165}$.

\item Any other knot of crossing number at most 10 has transient number one.
\end{enumerate}
\end{theorem}

\begin{proof} According to the information given in KnotInfo \cite{Knot}, the knots in (1) and (2) are precisely the knots
with crossing number up to 10, whose unknotting number and tunnel number are both larger that 1. So, any other knot has unknotting number of tunnel number equal to 1, and then have transient number 1. The knots in (1) are precisely the knots 
whose double branched cover has non-cyclic first homology group, and furthermore these knots have tunnel number 2.
Therefore its transient number must be two. The knots in (2) have tunnel number two but  their double branched cover have cyclic first homology group, hence we cannot calculate the transient number yet.
\end{proof}

A similar result can be done for the knots of crossing number 11 or 12.

The following knots are interesting, for we use the homology of $p$-branched covers of a knot to determine the transient number.

\begin{theorem} The following knots have transient number $2$: $10_{99}$, $10_{123}$, $12a_{427}$, $12a_{435}$, 
$12a_{465}$, $12a_{466}$, $12a_{475}$, $12a_{647}$, $12a_{742}$, $12a_{801}$, $12a_{868}$, $12a_{975}$, 
$12a_{990}$, $12a_{1019}$, $12a_{1102}$, $12a_{1105}$, $12a_{1167}$, $12a_{1206}$, $12a_{1229}$, 
$12a_{1288}$, $12n_{518}$, $12n_{533}$, $12n_{604}$, $12n_{605}$, $12n_{642}$, $12n_{706}$, $12n_{840}$, $12n_{879}$, $12n_{888}$. 
\end{theorem}

\begin{proof} According to Theorem \ref{cubiertageneral-p}, if $K$ has $tr(K)=1$, then $rank(H_1(\Sigma_p[K]) \leq p+1$. Using this and the information given in KnotInfo \cite{Knot}, we show that these knots cannot have transient number one. 
As they have tunnel
number two, in fact must also have transient number two. Below in the Table, there is a list of the knots with the corresponding
homology group needed for the proof. For some of them, it is enough to use the homology of $\Sigma[K]$,
but not for all. A symbol $\{6,\{2,2,2,10,20,340,0,0\}\}$ means that 
$H_1(\Sigma_6[K])=\mathbb{Z}_2+\mathbb{Z}_2+\mathbb{Z}_2+\mathbb{Z}_{10}+\mathbb{Z}_{20}+\mathbb{Z}_{340}+\mathbb{Z}+\mathbb{Z}$.
\end{proof}


\begin{scriptsize}
\begin{center}
\begin{tabular}{|c|c|c|}

\hline
&$ 10_{99}$	& $\{2,\{9,9\}\},\quad \{6,\{2,2,6,6,0,0,0,0\}\}$ \\
\hline
& $10_{123}$ & 	$\{2,\{11,11\}\},\quad \{5,\{2,2,2,2,2,2,2,2\}\}$ \\
\hline
& $12a_{427}$	 & $\{2,\{15,15\}\},\quad  \{4,\{3,3,3,3,15,15\}\},\quad \{6,\{4,4,20,20,0,0,0,0\}\} $ \\
\hline
&$12a_{435}$	 & $\{2,\{3,75\}\},\quad \{6,\{2,2,8,200,0,0,0,0\}\}$ \\
\hline
& $12a_{465}$	& $\{6,\{2,2,2,2,2,2,38,9158\}\}$ \\
\hline
&$12a_{466}$ &	$\{6,\{2,2,2,2,2,2,26,5434\}\} $\\
\hline
& $12a_{475}$ &	$\{6,\{2,2,2,10,20,340,0,0\}\} $\\
\hline
& $12a_{647}$ & 	$\{2,\{3,51\}\},\{6,\{2,2,2,34,0,0,0,0\}\} $\\
\hline


&$12a_{868}$ &	$\{5,\{2,2,2,2,8,8,88,88\}\}$ \\
\hline
& $12a_{975}$	& $\{2,\{5,45\}\},\quad \{4,\{5,5,5,5,5,45\}\} $\\
\hline
&$12a_{990}$ & $\{2,\{3,75\}\}\quad \{6,\{2,2,8,200,0,0,0,0\}\} $ \\
\hline
&$12a_{1019}$	& $\{2,\{19,19\}\},\quad \{5,\{6,6,6,6,6,6,6,6\}\} $ \\
\hline
&$12a_{1102}$ &	 $\{6,\{2,2,2,2,2,2,112,34160\}\} $\\
\hline
& $12a_{1105}$	& $\{2,\{17,17\}\},\quad \{6,\{2,2,2,2,10,10,170,170\}\} $ \\
\hline
&$12a_{1167}$ &	$\{5,\{2,2,2,2,2,2,82,82\}\} $\\
\hline

& $12a_{1229} $&	$\{5,\{2,2,2,2,8,8,8,8\}\}$ \\
\hline
& $12a_{1288}$	& $\{2,\{3,39\}\},\quad \{6,\{2,2,2,26,0,0,0,0\}\}$ \\
\hline
&$12n_{518}$ &	$\{2,\{3,21\}\}, \quad \{6,\{2,2,4,28,0,0,0,0\}\} $ \\
\hline
&$12n_{533}$	& $\{6,\{2,2,2,2,2,42,0,0\}\} $\\
\hline
&$12n_{604}$ &	$\{2,\{3,27\}\},\quad \{6,\{2,2,2,18,0,0,0,0\}\} $ \\
\hline
& $12n_{605}$	& $\{2,\{3,3\}\},\quad \{6,\{2,2,2,2,0,0,0,0\}\} $ \\
\hline

&$12n_{706}$	&  $\{2,\{7,7\}\},\quad \{5,\{3,3,3,3,3,3,3,3\}\},\quad \{6,\{2,2,2,2,2,2,14,14\}\}$ \\
\hline
&$12n_{840} $ &	$\{6,\{2,2,2,2,2,2,10,1190\}\}$ \\
\hline
&$12n_{879}$ & 	$\{5,\{2,2,2,2,4,4,4,4\}\} $ \\
\hline
&$12n_{888}$ &	$\{2,\{3,15\}\},\quad \{6,\{2,2,2,10,0,0,0,0\}\} $ \\
\hline

\end{tabular}
\end{center}
\end{scriptsize}

\section{Transient number and connected sums}\label{connectedsums}

It is natural to consider the behavior of a knot invariant with respect to connected sums. It is easy to see that $u(K_1 \# K_2) \leq u(K_1) + u(K_2)$, and the equality is conjectured to happen. It is also not difficult to see that $t(K_1 \# K_2)\leq t(K_1)+t(K_2)+1$. There are known examples of knots with 
$t(K_1 \# K_2) = t(K_1)+t(K_2)+1$ \cite{MSY}, examples with $t(K_1 \# K_2) = t(K_1)+t(K_2)$, and examples with 
$t(K_1 \# K_2) < t(K_1)+t(K_2)$ \cite{M}. So, we can expect a similar inequality for the transient number. 

\begin{theorem} Let $K_1$, $K_2$ be knots in $S^3$. Then $tr (K_1\# K_2) \leq tr (K_1)+tr (K_2) + 1$. \end{theorem}

\begin{proof} Let $K_1$ be a knot with transient number $tr(K)=n$, and let $\{ \gamma_1,\gamma_2,\dots,\gamma_n\}$ 
be a system of transient arcs for $K_1$. Let $T_1=\mathcal{N(}K\cup \gamma_1 \cup \gamma_2 \dots \cup \gamma_n)$. 
Then $T_1$ is a genus $n+1$ handlebody with the property that $K_1$ can be homotoped in the interior of $T_1$ to the 
trivial knot in $S^3$. We can assume that the homotopy that transform $K_1$ into the trivial
knot can be realized by a sequence of ambient isotopies of $K_1$ and crossing changes. So, suppose that after making
isotopies, all crossing changes are performed simultaneously. Suppose $r$ crossing changes are performed,
numbered $1,2,\dots,r$, and for each crossing change let $\alpha_i$ be an arc with endpoints in $K_1$
which remembers the crossing change, that is, if $B_i$ is a regular neighborhood of $\alpha_i$, in fact a 3-ball that intersects $K_1$ in two unknotted arcs, then a crossing change can be performed 
inside each $B_i$ to get the trivial knot.
Make an isotopy to move $K_1$ to its original position,
and then $\{\alpha_1,\alpha_2,\dots,\alpha_r\}$ is a collection of disjoint arcs with endpoints in $K_1$ contained  in $T_1$. 
Let $\delta_1$ be an arc in $T_1$ with an endpoint in $K_1$ and the other in $\partial N(K_1)$, such that $\delta_1$ is disjoint from the arcs $\alpha_i$. 

If $K_2$ is knot with $tr(K_2)=m$, then as above there is a genus $m+1$ handlebody that is the neighborhood
of $K$ union a system of transient arcs $\{ \gamma'_1,\gamma'_2,\dots,\gamma'_m\}$, and there is a collection of 
arcs $\{ \beta_1,\dots,\beta_s\}$ that determines
crossing changes that unknot $K_2$. Let $\delta_2$ be an arc in $T_2$ with an endpoint in $K_2$ and the other in 
$\partial N(K_2)$, such that $\delta_2$ is disjoint from the arcs $\beta_i$.

Suppose that $T_1$ and $T_2$ lie in disjoint 3-balls $C_1$ and $C_2$ contained in $S^3$. Suppose that
 $\partial T_i \cap \partial C_i$ consists of a disk $D_i$, such that the endpoint of $\delta_i$ lying in $\partial T_i$, 
 it lies in $D_i$, for $i=1,2$. Do a disk sum of $T_1$ and $T_2$, identifying $D_1$ and $D_2$, such that the endpoints of 
 $\delta_1$ and $\delta_2$ coincide. Let $\delta=\delta_1 \cup \delta_2$,
this is an arc with an endpoints in $K_1$ and $K_2$. Following $\delta$, do a band sum of $K_1$ and $K_2$.
As $K_1$ and $K_2$ lie in disjoint 3-balls, this band sum is in fact a connected sum $K_1\# K_2$. Let
$T=T_1\cup T_2$, this is a genus $n+m+2$ handlebody, and $K_1\# K_2$ can be homotoped to the trivial knot
inside it, to see that just do crossing changes following the arcs $\alpha_i$ and $\beta_j$. Now note that $T$
is the regular neighborhood of $K_1\# K_2$ and a system of $n+m+1$ arcs, that is, the $n$ arcs for a
system fo $K_1$, the $m$ arcs for a system of $K_2$, plus one more arc which is dual to the band used to
perform the connected sum of $K_1$ and $K_2$, see Figure \ref{connectedsum}. 
This shows that the transient number of $K_1\# K_2$ is at most $n+m+1$.
 \end{proof}
 
 \begin{figure}

\includegraphics[angle=0, width=8true cm]{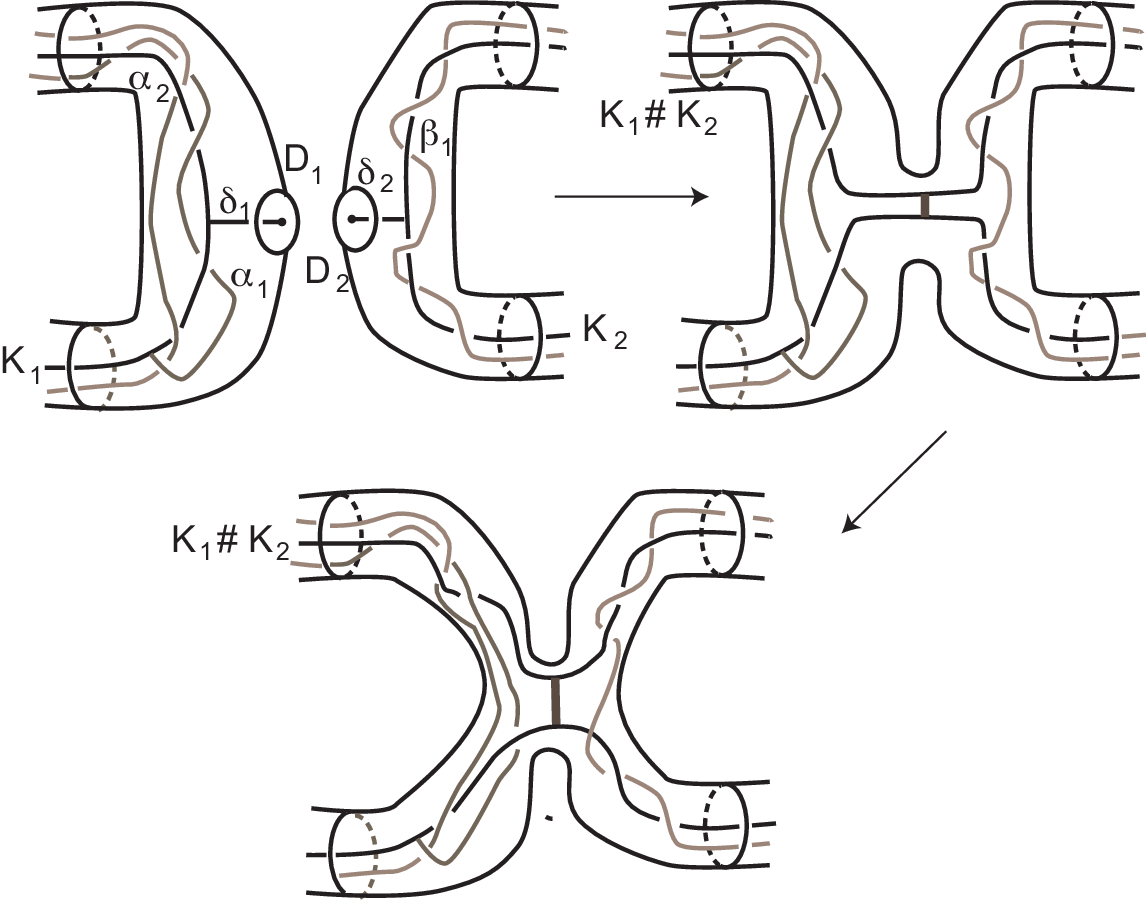}

\caption{}
\label{connectedsum}

\end{figure}
 In many cases we can ensure that $tr(K_1\# K_2)$ is at most $tr(K_1)+ tr(K_2)$. For example, if the arc systems
 that unknot $K_1$ and $K_2$ are disjoint from a meridian disk for $N(K_1)$ and a meridian disk for $N(K_2)$,
 then it can be shown that no more than $tr(K_1) + tr(K_2)$ arcs are needed to unknot $K_1\# K_2$. 
 
 There are examples of knots
 $K_1$, $K_2$, such that $t(K_1)=1=t(K_2)$, but $t(K_1\# K_2)=3$ \cite{MSY}. For these example, it is clear that $tr(K_1)=1=tr(K_2)$, but it is not clear what is $tr(K_1\# K_2)$.
 
 There are also examples of knots  $K_1$, $K_2$, such that $t(K_1)=2$, $t(K_2)=1$, but $t(K_1\# K_2)=2$ \cite{Mr}.
 In this case $tr(K_2)=1$ and $tr(K_1\# K_2)\leq 2$, but it is not clear whether $tr(K_1)=1$ or $2$. 
 
 It is well known that knots with unknotting number one or tunnel number one are prime, but the proofs are not so easy.
 The first proof that knots $K$ with $u(K)=1$ are prime \cite{S}, uses heavy combinatorial arguments, a second
 proof uses sutured manifold theory \cite{ST}, and a third proof depends on double branched covers and deep
 results on Dehn surgery on knots \cite{Z}. There are also two proofs that tunnel number one knots are prime,
 one uses combinatorial group theory \cite{N}, and other uses combinatorial arguments \cite{Sc}. 
 A proof that transient number one
 knots are prime would imply both, that unknotting number one and tunnel number one knots are prime, 
 so it may not be easy to prove that. However seems reasonable to conjecture the following.
 
 \begin{conjecture} If $K$ is a knot with $tr(K)=1$ then $K$ is prime.
 \end{conjecture}

Theorem \ref{largetransitnumber} (1) gives some evidence  for this conjecture.
 
\vskip30pt

\textbf{Acknowledgments.}
 This research was supported by a grant from the National Autonomous University of Mexico, UNAM-PAPIIT IN116720.

\end{document}